\documentclass[12pt]{amsart}

\usepackage{color}

\usepackage{amsmath,amssymb,setspace,nicefrac, yhmath, amscd,eucal}
\usepackage[active]{srcltx}
\usepackage[colorlinks, linkcolor=red, citecolor=blue, urlcolor=blue, hypertexnames=true]{hyperref}
\usepackage{amsrefs, mathrsfs}

\allowdisplaybreaks
\usepackage[all]{xy}

\newcommand{\R}{{\mathbb R}}
\newcommand{\Q}{{\mathbb Q}}
\newcommand{\Z}{{\mathbb Z}}
\newcommand{\N}{{\mathbb N}}

\newcommand{\cO}{{\mathcal O}}

\newcommand{\cG}{{\mathcal G}}

\newcommand{\cH}{{\mathcal H}}

\newtheorem{thm}{Theorem}[section]
\newtheorem{cor}[thm]{Corollary}
\newtheorem{lem}[thm]{Lemma}
\newtheorem{question}[thm]{Question}
\newtheorem{definition}[thm]{Definition}
\newtheorem{example}[thm]{Example}
\newtheorem{remark}[thm]{Remark}
\newtheorem{proposition}[thm]{Proposition}

\theoremstyle{definition}

\begin{document}
\title[Divergence, Undistortion and H\"{o}lder Cont. Cocycle Superrigidity]{Divergence, Undistortion and H\"{o}lder Continuous Cocycle Superrigidity for Full Shifts}
\author{Nhan-Phu Chung}
\address{Nhan-Phu Chung, Department of Mathematics, Sungkyunkwan University, Suwon 440-746, Korea. Tel: +82 031-299-4819} 
\email{phuchung@skku.edu;phuchung82@gmail.com}
\author{Yongle Jiang}
\address{Yongle Jiang, Department of Mathematics, Sungkyunkwan University, Suwon 440-746, Korea.}
\email{yongleji@buffalo.edu}

\date{\today}
\maketitle
\begin{abstract}
In this article, we will prove a full topological version of Popa's measurable cocycle superrigidity theorem for full shifts \cite{Popa2}. More precisely, we prove that every H\"{o}lder continuous cocycle for the full shifts of every finitely generated group $G$ that has one end, undistorted elements and sub-exponential divergence function is cohomologous to a group homomorphism via a H\"{o}lder continuous transfer map if the target group is complete and admits a compatible bi-invariant metric. Using the ideas of Behrstock, Dru\c tu, Mosher, Mozes and Sapir \cite{bd,BDM,dms,ds}, we show that the class of our acting groups is large including wide groups having undistorted elements and one-ended groups with strong thick of finite orders. As a consequence, irreducible uniform lattices of most of higher rank connected semisimple Lie groups, mapping class groups of $g$-genus surfaces with $p$-punches, $g\geq 2, p\geq 0$; Richard Thompson groups $F,T,V$; $Aut(F_n)$, $Out(F_n)$, $n\geq3$; certain (2 dimensional)-Coxeter groups; and one-ended right-angled Artin groups are in our class. This partially extends the main result in \cite{CJ}.
\end{abstract}

\section{introduction}
During studying measurable orbit equivalence, Popa established his celebrated cocycle superrigidity theorem for Bernoulli shifts in the measurable setting \cite{Popa2, Popa3}: for certain groups $G$ the full shift action $G\curvearrowright \prod_G(A,\mu)$ is $\mathscr{U}_{fin}$-measurable cocycle superrigid, for every finite set $A$, where $\mathscr{U}_{fin}$ is the class of Polish groups which arise as closed subgroups of
the unitary groups of finite von Neumann algebras. The class $\mathscr{U}_{fin}$ contains all countable groups. In \cite{CJ} we proved a topological version of Popa's theorem by showing that for every non-torsion finitely generated group $G$ with one end, every full shift $G\curvearrowright A^G$ is continuous $H$-cocycle rigid for every countable group $H$. Shortly after that, Cohen \cite{cohen} removed the ``non-torsion" assumption in \cite{CJ}. As one expects, this continuous cocycle superrigidity theorem has applications in continuous orbit equivalence theory, which was systematically studied in \cite{xinli}.

As Popa's theorem holds for the class $\mathscr{U}_{fin}$ groups containing all countable groups, one may wonder whether we can extend the main result in \cite{CJ} to other target groups belonging to $\mathscr{U}_{fin}$. More precisely, we study continuous cocycles for shifts with target groups being any Polish groups that admit compatible bi-invariant metrics. And following \cite[2.b.]{Furman}, we denote this class of groups by $\mathscr{G}_{inv}$, which contains the class $\mathscr{U}_{fin}$.  

On the other direction, in 1994, in a seminal paper \cite{KatokSpatzier}, Katok and Spatzier established their pioneer rigidity properties of hyperbolic actions of higher rank abelian groups $\Z^k$, $k\geq 2$. They showed that for standard $\Z^k$-hyperbolic actions, $k\geq 2$ including all known irreducible examples, every H\"{o}lder continuous cocycle into $\R^l$ is H\"{o}lder cohomologous to a constant cocycle without any assumptions on periodic data as  Liv\v sic's theory for $\Z$-actions. After that, many results of rigidity for cocycles in higher rank abelian group actions have been proved, see the recent monograph \cite{KatokNitica} for details and references. As a nature, we would like to investigate whether   
results of \cite{KatokSpatzier} can be extended for hyperbolic actions of groups beyond $\Z^k$. This paper is our first steps to understand this question when we consider our ``simplest Anosov action'', the full shifts of general countable groups, and the target groups of cocycles are complete and admits compatible bi-invariant metrics. Following \cite{katok_schmidt, Schmidt95}, we investigate the following class of groups:
\begin{equation*}
\begin{split}
\cG_H &:=\{G|~ \mbox{Every H\"{o}lder continuous cocycle~ $c: G\times A^G\to H$ is H\"{o}lder cohomologous}\\
& \mbox{to a group homomophism from $G$ to $H$ for all $H\in \mathscr{G}_{inv}$ and all finite set $A$}\}.
\end{split}
\end{equation*}
Since any continuous cocycles into discrete target groups are automatically H\"{o}lder continuous, we have $\cG_H\subseteq \cG$, where $\cG$ is defined in \cite{CJ} as follows:
\begin{equation*}
\begin{split}
\cG &:=\{G|~ \mbox{Every continuous cocycle~ $c: G\times A^G\to H$ is continuous cohomologous}\\
& \mbox{to a group homomophism from $G$ to $H$ for all countable discrete group $H$}\\
& \mbox{ and all finite set $A$}\}.
\end{split}
\end{equation*}
We proved in \cite{CJ} that if $G$ is finitely generated and non-torsion, then $G$ belongs to $\cG$ iff $G$ has one end. We expect the same conclusion still holds if we replace $\cG$ with $\cG_H$. Along this direction, we prove the following theorem, which is a counterpart of \cite[Theorem 1]{CJ} in our setting. 

\begin{thm}\label{main thm in 2nd approach}
Let $A$ be a finite set and $G$ be a finitely generated infinite group.  Let $\sigma: G\curvearrowright X\subseteq A^G$ be a subshift. Assume 
the following conditions are satisfied:

(1) $G$ has one end and its divergence function grows sub-exponentially (see Section \ref{section on divergence}).

(2) $G$ has slow growth distortion property (abbreviated as (SDT) property) (see Definition \ref{def of SDT groups}).

(3) $\sigma$ is topological mixing, and the homoclinic equivalence relation $\Delta_X$ has strong $a$-specification (see Section \ref{section on specification property}) for every element $a$ in $G$ with (SDT) property. 

Then every H\"{o}lder continuous cocycle $c:G\times X\to H$ on any group $H\in\mathscr{G}_{fin}$ is cohomologous to a homomorphism via a continuous transfer map. If we further assume $G$ has undistorted elements (see Section \ref{section on undistorted elements}), then this transfer map can be chosen to be H\"{o}lder continuous.
\end{thm}

From this theorem, we deduce the following:
\begin{cor}\label{corollary on the membership of G_h}
Let $G$ be a finitely generated infinite group. If $G$ has one end, undistorted elements and its divergence function grows sub-exponentially. Then $G\in\cG_H$.
\end{cor}

We organize our paper as follows. In section 2, we review definitions and elementary properties of divergence functions, distortion property, undistorted elements and translation numbers. In section 3, we review H\"{o}lder continuous cocycle, homoclinic equivalence relation for shifts and also introduce a modified version of specification property for shifts to deal with rigidity of H\"{o}lder continuous cocycles. In this section, we also establish a key lemma about construction of invariant holonomies for undistorted elements.   
In section 4, we prove our main theorem by combining our techniques in \cite{CJ} with our sub-exponential growth rate of divergence functions to get a universal solution of continuous transfer map and show that it is indeed H\"{o}lder continuous whenever we have undistorted elements. Finally, using ideas in \cite{BDM,bd,dms,ds} we illustrate that our class $\cG_H$ contains wide groups having undistorted elements and one-ended groups with strong thick of finite orders. This implies that irreducible uniform lattices of most of higher rank connected semisimple Lie groups, mapping class groups of $g$-genus surfaces with $p$-punches, $g\geq 2,p\geq 0$; R. Thompson groups $F,T,V$; $Aut(F_n)$, $Out(F_n)$, $n\geq3$; certain (2-dimensional) Coxeter groups; and one-ended right-angled Artin groups are in $\cG_H$. 

\textbf{Acknowledgement:} N.-P. Chung was supported by the NRF grant funded
by the Korea government (MSIP) (No. NRF-2016R1D1A1B03931922). Y. Jiang is grateful to his advisor Prof. Hanfeng Li for his support in Buffalo, where part of this paper was written. Both authors were supported by Science Research Center Program through NRF funded by the Ministry of Science, ICT \& Future Planning (No. NRF-2016R1A5A1008055). We thank Jason Behrstock for informing us examples in \cite {OOS} of finitely generated groups whose divergences are sub-exponential but non-polynomials. We would also like to thank Mark Sapir for helpful comments, specially providing us more information of R. Thompson groups $F,T,V$. 
\section{Preliminaries}

In this section, we recall some definitions and set up notations. 
Throughout this paper, $G$ is a finitely generated, non-torsion group unless otherwise specified.

\subsection{Divergence functions}\label{section on divergence}
Let $(X, d)$ be a geodesic metric space, typically, we would take $X$ to be a Cayley graph of $G$ endowed with the word metric with respect to some symmetric generating set. The main reference for divergence functions we use are \cite[section 3]{bd} and \cite[section 3.1]{dms}. Note that there are several definitions of divergence functions, here we use \cite[Definition 3.1]{dms} by taking $\delta=\frac{1}{2}$ and $\gamma=2$ since these two constants would guarantee these definitions are equivalent by \cite[Corollary 3.12(2)]{dms} when $X$ is a Cayley graph of a finitely generated one-ended group.

We define the \emph{divergence of a pair of points $a, b\in X$ relative to a point $c\not\in\{a, b\}$} as the length of the shortest path from $a$ to $b$ avoiding a ball around $c$ of radius $d(c, \{a, b\})/2-2$, where $d(c, \{a, b\}):=\min\{d(c, a), d(c, b)\}$. If no such path exists, then we define the divergence to be infinity. The \emph{divergence of a pair of points $a, b$}, $Div(a, b)$, is the supremum of the divergences of $a, b$ relative to all $c\in X\setminus \{a, b\}$.

\emph{The divergence of $X$} is given by $Div_X(n):=\max\{Div(a, b)~|~ d(a, b)\leq n\}$. We say $Div_X(\cdot)$ \emph{grows sub-exponentially} if $\lim_{n\to\infty}Div_X(n)r^n=0$ for all $0<r<1$.

One can check $\lim_{n\to\infty}Div_X(n)r^n=0$ for all $0<r<1$ iff $\lim_{n\to\infty}\log(Div_X(n))/n=0$. And if $Div_X(\cdot)$ is a polynomial, then it grows sub-exponentially.

Recall that given two non-decreasing functions $f, g: \R_+\to \R_+$, we write $f\preceq g$ if there exists a constant $C\geq 1$ for which $f(x)\leq Cg(Cx+C)+Cx+C$ for all $x\in \R_+$; we set $f\asymp g$ if both $f\preceq g$ and $g\preceq f$.

It is easy to check that the $\asymp$-equivalence class of the divergence functions is a quasi-isometry invariant when $X=Cayley(G)$ and $d$ is a word metric. It is also clear that the property of having sub-exponential growing divergence function is preserved under quasi-isometry. 

\subsection{Distortion and (SDT) property} \label{(SDT) property}
We first recall the definition of distortion functions, which is needed for the definition of (SDT) property.

\begin{definition}[cf. \cite{gromov_book on asymptotic invariants}]
Let $K$ be a subgroup of a group $G$, where $K$ and $G$ are generated by finite sets $T$ and $S$, respectively. Then the \emph{distortion function} of $K$ in $G$ is defined as 
$\Delta_{K}^G: \R_{+}\to \R_{+},~ x\mapsto \max\{\ell_T(k):~ k\in K, \ell_S(k)\leq x\}.$ And we also define the \emph{compression function} $\rho_K^G$ as follows: $\rho_{K}^G: \R_{+}\to \R_{+},~ x\mapsto \min\{\ell_S(k):~ k\in K, \ell_T(k)\geq x\}.$
\end{definition}

For later use, we would frequently work with $\rho_K^G$ rather than $\Delta_K^G$. Roughly speaking, one can think of $\rho_K^G$ as the inverse function of $\Delta_K^G$ (see the proposition below). Note that we could not define the inverse function of $\Delta_K^G$ directly since $\Delta_K^G$ and $\rho_K^G$ may not be strictly monotone. However, we still use $(\rho_K^G)^{-1}(c)$ to denote the number $\sup\{\lambda>0: \rho_K^G(\lambda)\leq c\}$. This makes sense by Proposition \ref{property of distortion functions} (1) below. 

Given two non-decreasing functions $f, g: \R_+\to \R_+$, we write $f\approx g$ if there exists $c>0$ such that for $f(x)\leq cg(cx)$ and $g(x)\leq cf(cx)$ for all $x\in \R_+$. Let us record properties of these functions for later use. We leave the proof as exercises.

\begin{proposition}\label{property of distortion functions}
(1) Both $\Delta_K^G$ and $\rho_K^G$ are constant on open intervals $(n, n+1)$ for all $n\in \mathbb{N}$, non-decreasing and do not depend on the choice of finite generating sets $T$ and $S$ up to $\approx$ and hence also up to $\asymp$.

(2) $\Delta_K^G(\ell_S(k))\geq \ell_T(k)$ and $\rho_K^G(\ell_T(k))\leq \ell_S(k)$ for all $k\in K$. In particular, if $K=\langle g\rangle$ and $g$ is of infinite order, then $\Delta_K^G(\ell_S(g^i))\geq i$ and $\rho_K^G(i)\leq \ell_S(g^i)$ for all $i\geq 1$ by taking $T=\{g^{\pm}\}$.

(3) $\Delta_K^G(\rho_K^G(x)-1)<x<\rho_K^G(\Delta_K^G(x)+1)$ and $(\rho_K^G)^{-1}(x)\leq \Delta_k^G(x)$ for all $x\in \R_+$.

(4) If $K$ is infinite, then $\Delta_K^G$ is at least and $\rho_K^G$ is at most linear up to $\approx$. 

(5) If $K=\langle g\rangle$, $T=\{g^{\pm}\}$, where $g\in G$ has infinite order and $S=S^{-1}$. Then $\rho_K^G$ is subadditive, i.e. $\rho_K^G(x+y)\leq \rho_K^G(x)+\rho_K^G(y)$; $\Delta_K^G$ is superadditive, i.e. $\Delta_K^G(x+y)\geq \Delta_K^G(x)+\Delta_K^G(y)$ for all $x, y\in \R_+$. 
\end{proposition}

\begin{definition}\label{def of SDT groups}
Let $G$ be a finitely generated infinite group. We say $G$ has \emph{slow growth distortion property (abbreviated as (SDT) property)} if there exists some non-torsion element $g$ in $G$ such that $I(r):=\sum_{i=1}^{\infty}r^{\rho_{\langle g\rangle}^{G}(i)}<\infty$ for all $0<r<1$. And we also say such a $g$ has \emph{(SDT) property} or $g$ is an \emph{(SDT) element}. 
\end{definition}

We record the following facts for later use.

\begin{proposition}\label{property related to def of slow growth distortion}
(1) If a function $f: \R_+\to\R_+$ is non-decreasing and continuous except at at most countably many points. Then for every $r>0$, $\sum_{i=1}^{\infty}r^{f(i)}<\infty$ if and only if $\int_1^{\infty}r^{f(x)}dx<\infty$. 

(2) Let $f, g: \R_+\to\R_+$ be functions as in (1). If $f\approx g$, then $\sum_{i=1}^{\infty}r^{f(i)}<\infty$ for all $r\in (0, 1)$ if and only if $\sum_{i=1}^{\infty}r^{g(i)}<\infty$ for all $r\in (0, 1)$.

(3) If $f(x)=x^{1/m}$, where $m\geq 1$, then $\sum_{i=1}^{\infty}r^{f(i)}<\infty$ for all $r\in (0, 1)$.

(4) If there exist some $c>0, m\geq 1$ such that $\Delta_K^G(x)\leq c(cx)^m$ for all $x>0$, then $I(r):=\sum_{i=1}^{\infty}r^{\rho_K^{G}(i)}<\infty$ for all $r\in (0, 1)$.

(5) More generally, if $\Delta_K^G(\cdot)$ grows sub-exponentially, i.e. $\lim_{n\to\infty}\log(\Delta_K^G(n))/n=0$, then $I(r):=\sum_{i=1}^{\infty}r^{\rho_K^{G}(i)}<\infty$ for all $0<r<1$.
\end{proposition}
\begin{proof}
The proof of (1) and (2) are left as exercises.

To prove (3), observe that for any $k\geq 1$, $\#\{i\in \mathbb{N}: k\leq f(i)<k+1\}=\#\{i\in\mathbb{N}: k\leq i^{1/m}<k+1\}\leq (1+k)^m-k^m+1$.
Hence, $\sum_{i=n}^{\infty}r^{f(i)}\leq \sum_{k\geq f(n)}((1+k)^m-k^m+1)r^k\to 0$ as $n\to\infty$.

To prove (4), by Proposition \ref{property of distortion functions} (3), we have $x<\rho_K^G(\Delta_K^G(x)+1)\leq \rho_K^G(c(cx)^m+1)$. Hence, $\rho_K^G(y)\geq ((y-1)/c^{m+1})^{1/m}$. Then apply part (3).  

The proof of (5) is similar to the above proof of (4).
\end{proof}

In this paper, we are mainly interested in the case $K=\langle g\rangle$ for some non-torsion element $g$ in $G$. For later use, when there is no danger of confusion, we would simply write $\rho_g(x)$ (respectively, $\Delta_g(x)$) for $\rho_{\langle g\rangle}^G(x)$ (respectively, $\Delta_{\langle g\rangle}^G(x)$) with $T=\{g^{\pm}\}$.

\subsection{Undistorted elements and translation numbers}\label{section on undistorted elements}

To give explicit examples of groups with (SDT) property, let us first look at the special case when there exists some non-torsion element $g$ with $\Delta_g(\cdot)$ grows at most linearly. In general, recall the following well-known definition.

\begin{definition}
Let $G$ be a finitely generated infinite group and $K$ be a finitely generated subgroup of $G$. We say $K$ is \emph{undistorted} in $G$ if there exists some constant $c>0$ such that $\Delta_K^G(n)\leq cn$ for all $n\geq 1$.
\end{definition}
Another way to characterize undistorted elements is to use translation numbers that we recall below.

Fix any finite generating set $S$ of $G$, let $g\in G$ be an element of infinite order. As $\ell_S(g^{m+n})\leq \ell_S(g^m)+\ell_S(g^n)$ for every $m,n\in \N$, the limit $\lim_{n\to\infty}\ell_S(g^n)/n$ exists and equals to $\inf_{n\in \N}\ell_S(g^n)/n$ \cite[Theorem 4.9]{walters}. The \textit{translation number} of $g$ with respect to $S$ is defined by $\tau_S(g):=\lim_{n\to \infty}\ell_S(g^n)/n$ \cite[Section 6]{GS}. Some elementary properties of translation numbers are established in \cite[Lemma 6.3]{GS}. An element $g\in G$ is called  \textit{an undistorted element} if $\tau_S(g)>0$ for some (and every) finite generating set $S$ of $G$ \cite[Definition 1.1]{FH}.
We say that a finitely generated group $G$ satisfies \emph{property (UD)} if there exists an element $g\in G$ and a finite generating set $S$ such that $\ell_S(g^k)\geq k$ for every $k\in \N$. As $\tau_S(g^m)=|m|\tau_S(g)$ for every $g\in G, m\in \Z$, it is clear that a group $G$ has property (UD) iff it has an undistortion element $g$, i.e $\tau_S(g)\neq 0$ for some (and every) finite generating set $S$ of $G$.

We have the following proposition.
\begin{proposition}\label{equivalence of undistorted g, linear (inverse) distortion function}
Let $G$ be a finitely generated infinite group with a finite symmetric  generating set $S$ and $g$ be an element in $G$ with infinite order. The following statements are equivalent.\\
(1) $\langle g\rangle$ is undistorted in $G$, i.e. $\Delta_g(n)\leq cn$ for some $c>0$ and all $n\geq 1$.\\
(2) $g$ is an undistorted element in $G$, i.e. $\tau_S(g)>0$.\\
(3) There exists some positive constant $\lambda$ such that $\ell_S(g^n)\geq \lambda n$ for all $n\geq 1$.\\
(4) There exist some positive constants $\lambda, \mu$ such that $\ell_S(g^n)\geq \lambda n-\mu$ for all $n\geq 1$.\\
(5) There exists some positive constant $\lambda$ such that $\rho_g(n)\geq \lambda n$ for all $n\geq 1$.\\
(6) There exists some positive constants $\lambda, \mu$ such that $\rho_g(n)\geq \lambda n-\mu$ for all $n\geq 1$.\\
\end{proposition}
\begin{proof}
$(1)\Rightarrow (2)$. From Proposition \ref{property of distortion functions} (2) (3), we deduce $n<\rho_g(\Delta_g(n)+1)\leq \rho_g(cn+1)\leq \ell_S(g^{cn+1})$. Hence $\frac{\ell_S(g^{cn+1})}{cn+1}\geq\frac{n}{cn+1}$. Then take $\lim$ on both sides.

$(2)\Rightarrow (1), (2)\Rightarrow(3)\Rightarrow(4)\Rightarrow(2), (3)\Leftrightarrow(5)\Leftrightarrow(6)$. All these are clear.
\end{proof}

\begin{lem}
\label{compare lemma 8 in CJ}
Let $G$ be a finitely generated infinite group with (SDT) property, $G_{\infty}'$ be the subgroup generated by all elements with (SDT) property in $G$ and $G_{\infty, 1}'$ be the subgroup generated by all elements that are undistorted, then both $G_{\infty}'$ and $G_{\infty, 1}'$ are infinite normal subgroup of $G$.
\end{lem}
\begin{proof}
By definition, any (SDT) element has infinite order, hence $G_{\infty}'$ and $G_{\infty, 1}'$ are infinite. Now, for any $h\in G$ and any non-torsion element $g$ in $G$, $\ell_S(tg^nt^{-1})\geq \ell_S(g^n)-2\ell_S(t)$, hence $\rho_{tgt^{-1}}(n)\geq \rho_{g}(n)-2\ell_S(t)$ for all $n\in \mathbb{N}$, where $S$ is any finite generating set of $G$. Hence $tgt^{-1}$ has (SDT) property (respectively, is undistorted) if $g$ has the same property, so $G_{\infty}'$ and $G_{\infty, 1}'$ are normal subgroups of $G$.
\end{proof}
\section{H\"{o}lder continuous cocycles, homoclinic equivalence relation and specification property for shifts}
\subsection{H\"{o}lder continuous cocycles and homoclinic equivalence relation}\label{section on Holder cocyles}
Let $(G, X)$ be a subshift, where $X\subset A^G$ for some finite set $A$, $S$ be a finite symmetric generating set of $G$. We write $B(n):=\{g\in G: \ell_S(g)\leq n\}$, where $\ell_S$ be the word length on $G$ induced by $S$. Following \cite[p.245]{Schmidt95}, a function $f:X\to H$ is called \emph{H\"{o}lder continuous} if there exist $C>0$ and $0<r<1$ such that for every $n\geq 0$, $d(f(x),f(y))\leq Cr^n$, for every $(x,y)\in X\times X$ with $x_{B(n)}=y_{B(n)}$. A cocycle $c: G\times X\to H$ is \emph{H\"{o}lder continuous} if $c(g,\cdot): X\to H$ is H\"{o}lder continuous for every $g\in G$. We say two H\"{o}lder continuous cocycles $c_1, c_2: G\times X\to H$ are \emph{H\"{o}lder cohomologous} if $c_1(g, x)=b(gx)^{-1}c_2(g, x)b(x)$ holds for all $g\in G$ and all $x\in X$ for some H\"{o}lder continuous map $b: X\to H$.  

Let us comment on this definition. In general, for two metric spaces $(X, d), (Y, d')$, a function $f: X\to Y$ is H\"{o}lder continuous if for some positive constants $c, \alpha$, $d'(f(x_1), f(x_2))\leq Cd(x_1, x_2)^{\alpha}$ holds for all $x_i$ in $X$. But this definition depends on the specific metric $d$ rather than the topology of $X$, which is not satisfactory when dealing with a subshift $(G, X)$. Thus we need a meaningful definition of  H\"{o}lder continuous map that is independent of the specific metric on $X$, such a definition is given in \cite[p.109]{katok_schmidt} for $\Z^d$-subshift. Choosing $d(x_1, x_2):=(1/2)^{\sup\{n: ~(x_1)_{B(n)}=(x_2)_{B(n)}\}}$ in \cite[p.109]{katok_schmidt}, one can check the definition in \cite[p.109]{katok_schmidt} is reduced to the one in \cite[p.245]{Schmidt95}. 

Note that in the definition of H\"{o}lder continuous cocycles, for different $g$, the associated $C, r$ may be different.

As in \cite{CJ}, we would use $\Delta_X$ to denote the homoclinic equivalence relation for a subshift $(G, X)$.

Let $f:X\to H$ be a H\"{o}lder continuous map. For every $n\geq 1$, we define maps 
\begin{align*}
c_f^{(g),+,(n)}(x,y) & =\Big(\prod_{j=0}^{n-1}f(g^jx)^{-1}
\Big)\Big(\prod_{j=0}^{n-1}f(g^jy)^{-1}\Big)^{-1},\\
c_f^{(g),-,(n)}(x,y) & =\Big(\prod_{j=1}^{n-1}f(g^{-j}x)
\Big)\Big(\prod_{j=1}^{n-1}f(g^{-j}y)\Big)^{-1}.
\end{align*}
\begin{lem}\label{limit exists in c_f^g, +}
Assume that $g$ is an (SDT) element in $G$, $H$ is complete w.r.t. a compatible bi-invariant metric $d$, e.g. $H\in\mathscr{G}_{inv}$. Then the maps $c_f^{(g),+},~ c_f^{(g),-}:\Delta_X\to H$ defined by
\begin{align*} 
c_f^{(g),+}(x,y):=\lim_{m\to \infty}c_f^{(g),+,(m)}(x,y),\\
c_f^{(g),-}(x,y):=\lim_{m\to \infty}c_f^{(g),-,(m)}(x,y),
\end{align*}
 for every $(x,y)\in \Delta_X$ are well defined and satisfy the cocycle condition 
\begin{align*} 
 c_f^{(g),+}(x,y)c_f^{(g),+}(y,z)=c_f^{(g),+}(x,z),\\
 c_f^{(g),-}(x,y)c_f^{(g),-}(y,z)=c_f^{(g),-}(x,z),
 \end{align*} 
 for every $(x,y),(y,z)\in \Delta_X$. \\
Furthermore, there exists $C'>0$ such that for every $N\in \mathbb{N}$, if $(x,y)\in \Delta_X$ with $x_h=y_h$ for every $h\notin B(N)$ then 
\begin{eqnarray}\label{estimate of partial sum}
d(c_f^{(g),+}(x,y),c_f^{(g),+,(m+{\rho_g}^{-1}(N))}(x,y))\leq C'\sum_{j=\rho_g^{-1}(N)+m}^{\infty}r^{\rho_g(j)}
\end{eqnarray} for every $m\geq 0$; Here, $r$ is the constant appeared when defining H\"{o}lder continuity of $f$.
\end{lem}
\begin{proof}
For every $j>(\rho_g)^{-1}(N)$ and every $h\in B(\rho_g(j)-N-1)$, one has $\ell_S(g^{-j}h)\geq \ell_S(g^{-j})-\ell(h)\geq \rho_g(j)-(\rho_g(j)-N-1)=N+1$, hence $(g^jx)_{B(\rho_g(j)-N-1)}=x_{g^{-j}B(\rho_g(j)-N-1)}=y_{g^{-j}B(\rho_g(j)-N-1)}=(g^jy)_{B(\rho_g(j)-N-1)}$. Thus, for every $\rho_g^{-1}(N)<n<n'$, one has 
\begin{eqnarray*}
&&d(c_f^{(g),+,(n)}(x,y), c_f^{(g),+,(n')}(x,y))\\
&=&d(\Big(\prod_{j=0}^{n-1}f(g^jx)^{-1}
\Big)\Big(\prod_{j=0}^{n-1}f(g^jy)^{-1}\Big)^{-1}, \Big(\prod_{j=0}^{n'-1}f(g^jx)^{-1}
\Big)\Big(\prod_{j=0}^{n'-1}f(g^jy)^{-1}\Big)^{-1})\\
&=& d(1_H, \Big(\prod_{j=n}^{n'-1}f(g^jx)^{-1}
\Big)\Big(\prod_{j=n}^{n'-1}f(g^jy)^{-1}\Big)^{-1})\\
&=& d(\prod_{j=n}^{n'-1}f(g^jy)^{-1}, \prod_{j=n}^{n'-1}f(g^jx)^{-1}
)\\
&\leq &\sum_{j=n}^{n'-1}d(f(g^jy)^{-1}, f(g^jx)^{-1}) \\
&=& \sum_{j=n}^{n'-1}d(f(g^jx), f(g^jy))\\
&\leq & C\sum_{j=n}^{n'-1}r^{\rho_g(j)-N-1}
\leq  Cr^{-N-1}\sum_{j=n}^{\infty}r^{\rho_g(j)}.
\end{eqnarray*} 
Thus, $\{c_f^{(g),+,(n)}(x,y)\}_{n\in\N}$ is a Cauchy sequence since $g$ has (SDT) property. As $H$ is complete, $c_f^{(g),+}$ is well defined. Choose $C'=Cr^{-N-1}$, then we also get \eqref{estimate of partial sum} from the last inequality. Similarly, $c_f^{(g),-}$ is also well-defined. 
\end{proof}

\subsection{Specification property for shifts} \label{section on specification property}
For the proof of Theorem \ref{main thm in 2nd approach}, we need the following version of specification property for a subshift. This is one key novelty in this paper since to handle $\mathscr{G}_{inv}$-target (H\"{o}lder continuous) cocycles, we need to generalize the version of specification used in Schmidt's approach in \cite{Schmidt95}. In his definition, it involves some cone structure that is defined using the Euclidean structure of $\Z^d$, so to extend this definition to general groups actions, we need a suitable counterpart for the Euclidean structure. Note that we only need a degenerate cone (i.e. a line) structure when defining the specification property used in \cite{CJ}.

\begin{definition}
\label{D-specification}
Let $G$ be a finitely generated group with a finite generating set $S$. Let $a$ be an element in $G$. Write $r_j=\rho_a(j)/4$, for every $R\geq 0$, we define
$$\Delta^+(a,R):=\{(x,x')\in \Delta_X: x_{a^{k}B([r_k]+R)}=x'_{a^{k}B([r_k]+R)},~\mbox{for all $k\geq 0$}\},$$
$$\Delta^-(a,R):=\{(x,x')\in \Delta_X: x_{a^{-k}B([r_k]+R)}=x'_{a^{-k}B([r_k]+R)},~\mbox{for all $k\geq 0$}\}.$$
We say that the equivalence class $\Delta_X(\bar{x})$ of a point $\bar{x}\in X$ has \emph{strong $a$-specification} (respectively, \emph{$a$-specification}) if $\Delta_X(\bar{x})\cap a^{-1}\Delta_X(\bar{x})$(respectively, $\Delta_X(\bar{x})$) is dense in $X$, and if  there exist constants $s'\geq 1, t'\geq 0$ with the following property: for any $R\geq 0$, if $x,x'\in \Delta_X(\bar{x})$ satisfy that $x_{B(N)}=x'_{B(N)}$, where $N=\lceil s'\ell_S(a)\rho_a^{-1}(4R)+2R+t'\rceil,$ then we can find an element $y$ in $\Delta_X(\bar{x})$ such that $(x,y)\in \Delta^+(a, R)$ and $(x',y)\in \Delta^-(a, R)$. And we say that the homoclinic equivalence relation $\Delta_X$ has \emph{strong $a$-specification} (respectively, \emph{$a$-specification}) if there exists a point $\bar{x}\in X$ such that $\Delta_X(\bar{x})$ has strong $a$-specification (respectively, $a$-specification).
\end{definition}
\begin{remark}
In the above definition of $r_j$, one can take a different ratio than $1/4$ as long as it is less than $1/2$. This would be clear from the proof of Lemma \ref{cone intersection is small} below. For later use, we would take $a$ to be an (SDT) or undistorted element.
\end{remark}

We record the following lemmas for later use.
\begin{lem}
\label{cone intersection is small}
Let $g$ be an element in $G$ with infinite order. Then
$(\bigcup_{j\geq 0}g^{-j}B([r_j]+R))\cap (\bigcup_{j\geq 0}g^{j}B([r_j]+R))\subseteq B(\ell_S(g)\rho_g^{-1}(4R)+2R)$. 
\end{lem}
\begin{proof}
Suppose $h=g^{-j_1}h_1=g^{j_2}h_2$ for some $j_1, j_2\geq 0$ and $h_i\in B([r_{j_i}]+R)$, $i=1, 2$.
Then $g^{j_1+j_2}=h_1h_2^{-1}$; hence, 
$\rho_g(j_1+j_2)\leq \ell_S(g^{j_1+j_2})=\ell_S(h_1h_2^{-1})\leq [r_{j_1}]+[r_{j_2}]+2R\leq 2r_{j_1+j_2}+2R\leq 2\rho_g(j_1+j_2)/4+2R$,
therefore, $\rho_g(j_1+j_2)\leq 4R$, and $j_1+j_2\leq \rho_g^{-1}(4R)$, so $\ell_S(h)=\ell_S(g^{j_2}h_2)\leq \ell_S(g^{j_2})+\ell_S(h_2)\leq \ell_S(g)j_2+r_{j_2+j_1}+R\leq \ell_S(g)\rho_g^{-1}(4R)+\rho_g(j_1+j_2)/4+R\leq \ell_S(g)\rho_g^{-1}(4R)+2R$.
\end{proof}
\begin{lem}\label{full shifts have strong specification}
Let $(G, A^G)$ be the full shift. If $a\in G$ is an (SDT) element, then $\Delta_X(\bar{x})$ has strong $a$-specification, where $\bar{x}=(0)_G$, i.e. $\bar{x}$ is the element in $X=A^G$ with every coordinate to be a constant $0\in A$.
\end{lem}
\begin{proof}
The proof is similar to the proof of \cite[Lemma 4]{CJ}, but we use Lemma \ref{cone intersection is small} now instead of \cite[Lemma 3]{CJ}.
\end{proof}

\section{Proof of Theorem \ref{main thm in 2nd approach}}
Recall that the main steps in the proof of \cite[Theorem 1]{CJ} is to apply \cite[Lemma 2, Corollary 2, Lemma 5, Lemma 6, Lemma 7, Lemma 1]{CJ} successively. In this section, we prepare the corresponding lemmas in our setting. The proofs are natural modification of the ones in \cite{CJ}. For completeness, we record the main changes below.

\begin{lem}
\label{compare lemma 2 in CJ}
Let $G$ be a group with a finite symmetric generating set $S$. Assume $G$ has one end, sub-exponential divergence function and $g, h$ are (SDT) elements in $G$, then $c_{f_g}^{(g), +}(x, y)=c_{f_h}^{(h), +}(x, y)$ for all $(x, y)\in\Delta_X$, where $f_g(x):=c(g, x), ~f_h(x):=c(h, x)$ for all $x\in X$ and $c: G\times X\to H$ is a H\"{o}lder continuous cocycle with $(G, X)$ being a subshift and $H\in\mathscr{G}_{inv}$.
\end{lem}

\begin{proof}
First, since $\lim_{n\to\infty}\ell_S(g^n)=\infty$, after passing to a subsequence, we may assume $3\ell_S(g^n)/2<\ell_S(g^{n+1})$ for all $n\geq 1$. Then, for each large enough $n$ (e.g. any $n$ such that $\ell_S(g^n)>\ell_S(h)$), we may find $m(n)$ such that $2\ell_S(g^n)<\ell_S(h^{m(n)})<3\ell_S(g^{n})$. This is possible since $|\ell_S(h^{m})-\ell_S(h^{m+1})|\leq \ell_S(h)<3\ell_S(g^n)-2\ell_S(g^n)$.

Note that we have the following facts. 

(1) We may assume $\{m(n)\}_n$ is an increasing sequence after passing to a subsequence since $3\ell_S(g^n)<2\ell_S(g^{n+1})$. 

(2) $\ell_S(g^{-n}h^{m(n)})\geq \ell_S(h^{m(n)})-\ell_S(g^n)\geq \ell_S(g^n)\to\infty$ as $n\to\infty$.

(3) $\ell_S(g^{-n}h^{m(n)})\leq \ell_S(h^{m(n)})+\ell_S(g^n)\leq (1/2+1)\ell_S(h^{m(n)})$, also $\leq (1+3)\ell_S(g^n)$; hence $\ell_S(g^{-n}h^{m(n)})\leq 4\cdot\min\{\ell_S(g^n), \ell_S(h^{m(n)})\}$.

Since both $g$ and $h$ are (SDT) elements, by Lemma \ref{limit exists in c_f^g, +} and the definition of $c_f^{(g), +, (n)}$, we just need to show $$\lim_{m=m(n)\to\infty}\lim_{n\to\infty}d(c(h^m, x)^{-1}c(h^m, y), c(g^n, x)^{-1}c(g^n, y))=0.$$ 

Since $(x, y)\in \Delta_X$, write $x|_{B(R)^c}=y|_{B(R)^c}$ for some $R>0$.

Let $n>>1$. Since $G$ has one end, we can find a path $p=p(n)$ which avoids the ball with radius $\min\{\ell_S(g^{-n}), \ell_S(h^{-m(n)})\}/2-2$ to connect $g^{-n}$ with $h^{-m}$. Write $g^{-n}=h^{-m}s_1^{-1}\cdots s_k^{-1}$, where $k\leq Div_X(\ell_S(g^{-n}h^m))$ and $s_i\in S$. Then from the following equality:
$$c(g^n, x)=c(s_k, s_{k-1}\cdots s_1h^mx)c(s_{k-1}, s_{k-2}\cdots s_1h^mx)\cdots c(s_1, h^mx)c(h^m, x),$$
we deduce that 
\begin{eqnarray*}
&&d(c(h^m, x)^{-1}c(h^m, y), c(g^n, x)^{-1}c(g^n, y))\\
&=& d(c(s_k, s_{k-1}\cdots s_1h^mx)c(s_{k-1}, s_{k-2}\cdots s_1h^mx)\cdots c(s_1, h^mx),\\
&& c(s_k, s_{k-1}\cdots s_1h^my)c(s_{k-1}, s_{k-2}\cdots s_1h^my)\cdots c(s_1, h^my))\\
&\leq& \sum_{j=1}^k d(c(s_j, s_{j-1}\cdots s_1h^mx), c(s_j, s_{j-1}\cdots s_1h^my)).\\
\end{eqnarray*}
Now, note that $(s_{j-1}\cdots s_ih^mx)_{B(M)}=(s_{j-1}\cdots s_ih^my)_{B(M)}$ for all $1\leq j\leq k$, where 
$M=\min\{\ell_S(g^{-n}), \ell_S(h^{-m})\}/2-2-R-1\geq\ell_S(g^{-n}h^m)/8-R-3.$

Hence, we have the following estimation.
\begin{eqnarray*}
&&\sum_{j=1}^k d(c(s_j, s_{j-1}\cdots s_1h^mx), c(s_j, s_{j-1}\cdots s_1h^my))\\
&\leq & kCr^{M}\leq C\cdot Div_X(\ell_S(g^{-n}h^m))r^{\ell_S(g^{-n}h^m)/8-R-3}\\
&\to & 0, ~\mbox{as $n\to\infty$ since $\lim_{n\to\infty}\ell_S(g^{-n}h^m)=\infty$}.
\end{eqnarray*}
Here $r:=\max\{r_s: s\in S\}, C:=\max\{C_s: s\in S\}$, where $r_s, C_s$ are the constants appeared when defining H\"{o}lder continuity of $c(s, -)$.
\end{proof}
\begin{remark}
One can also use Gersten's definition of divergence functions in \cite{gersten_div} for the above proof (still assuming the sub-exponential growing condition on these functions). For the relation between this function and the one used in this paper, see \cite[Lemma 3.13]{dms}.
\end{remark}

Similar to \cite[Corollary 2]{CJ}, we have the following lemma.
\begin{lem}
\label{compare corollary 2 in CJ}
Let $G$ be a group with a finite symmetric generating set $S$. Assume $G$ has one end and $g$ is an (SDT) element in $G$, then $c_{f_g}^{(g), +}(x, y)=c_{f_g}^{(g), -}(x, y)$ for all $(x, y)\in\Delta_X$, where $f_g(x):=c(g, x)$ for all $x\in X$ and $c: G\times X\to H$ is a H\"{o}lder continuous cocycle with $(G, X)$ being a subshift and $H\in\mathscr{G}_{inv}$.
\end{lem}

\begin{lem}
\label{compare lemma 5 in CJ}
Let $X\subset A^G$ be a subshift and $f: X\to H$ be H\"{o}lder continuous, where $H\in\mathscr{G}_{inv}$. Assume that $\Delta_X(\bar{x})$ has $g$-specification for some point $\bar{x}\in X$, where $g\in G$ is an element with (SDT) property and the cocycles $c_f^{(g),\pm}:\Delta_X\to H$ in Lemma \ref{limit exists in c_f^g, +} are equal.
Then 

$\lim_{\substack{(x,x')\to\Delta \\ x,x'\in\Delta_X(\bar{x})}}d(c_f^{(g),+}(x,x'),e_H)=0$, where $\Delta\subset X\times X$ denotes the diagonal.
\end{lem}

\begin{proof}
Fix any $l>1, R>0$.

Now, fix $(x', x)\in \Delta_X$ and suppose $d(x', x)$ is sufficiently small such that $x_{B(L)}=x'_{B(L)}$, where $L\geq \lceil s'\ell_S(g)\rho_g^{-1}(4R)+2R+t'\rceil$ for some $s',t'$ as in Definition \ref{D-specification}, then $\Delta_X(\bar{x})$ has $g$-specification implies there exists some $y\in \Delta_X(\bar{x})$ with $(x', y)\in\Delta^{-}(g, R)$ and $(x, y)\in \Delta^+(g, R)$.

Since $(x', y)\in\Delta^{-}(g, R)$, $x', y$ satisfy that $(g^jx')_{B([r_j]+R)}=(g^jy)_{B([r_j]+R)}$ for all $j\geq 0$. Then we deduce the following.
\begin{eqnarray*}
d(\prod_{j=0}^{l-1}f(g^{j}x')^{-1}, \prod_{j=0}^{l-1}f(g^{j}y)^{-1})&\leq& \sum_{j=0}^{l-1}d(f(g^jx'), f(g^jy))
\leq\sum_{j=0}^{l-1} Cr^{[r_j]+R}\\
&= & Cr^R\sum_{j=0}^{l-1}r^{[r_j]}\leq r^{R} C\sum_{j\geq 0}r^{[r_j]}.
\end{eqnarray*}
Hence, $d(c_f^{(g), +}(y, x'), e_H)\leq r^{R}C\sum_{j\geq 1}r^{[r_j]}$.
Here, $r, C$ are the constants appeared when defining $f$ is H\"{o}lder continuous. We have used $d(f(g^jx'), f(g^jy))\leq Cr^{[r_j]+R}$ and $\sum_{j\geq 0}r^{[r_j]}<\infty$ (since $g$ has (SDT) property).

Next, since $(x, y)\in \Delta^+(g, R)$, $(g^{-j}x)|_{B([r_j]+R)}=(g^{-j}y)|_{B([r_j]+R)}$ for all $j\geq 0$. We deduce the following.
\begin{eqnarray*}
d(\prod_{j=1}^{l-1}f(g^{-j}x), \prod_{j=1}^{l-1}f(g^{-j}y))&\leq& \sum_{j=1}^{l-1}d(f(g^{-j}x), f(g^{-j}y))
\leq \sum_{j=1}^{l-1} Cr^{[r_j]+R}\\
&= & Cr^R\sum_{j=1}^{l-1}r^{[r_j]}
\leq r^{R}C\sum_{j\geq 1}r^{[r_j]}.
\end{eqnarray*}
Hence $d(c_f^{(g), -}(x, y), e_H)\leq r^{R} C\sum_{j\geq 1}r^{[r_j]}$.
Thus, we deduce 
\begin{eqnarray*}
d(c_f^{(g), +}(x, x'), e_H)&\leq& d(c_f^{(g), +}(x, x'), c_f^{(g), +}(x, y))+d(c_f^{(g), +}(x, y), e_H)\\
&=&d(c_f^{(g), +}(y, x'), e_H)+d(c_f^{(g), -}(x, y), e_H)\\
&\leq & 2Cr^R\sum_{j\geq 0}r^{[r_j]}\to 0, \mbox{as $R\to\infty$}.
\qedhere \end{eqnarray*}
\end{proof}

\begin{lem}
\label{compare lemma 6 in CJ}
Assume that there exists a point $\bar{x}\in X$ such that 
$\Delta_X(\bar{x})\cap g^{-1}\Delta_X(\bar{x})$ is dense in $X$, where $g\in G$ has (SDT) property, and $\lim_{\substack{(x,x')\to\Delta \\ x,x'\in\Delta_X(\bar{x})}}d(c_f^{(g),+}(x,x'),e_H)=0$, where $f: X\to H$ is H\"{o}lder continuous and $H\in\mathscr{G}_{inv}$, then there is a continuous map $b: X\to H$ such that the map $x\mapsto b(gx)^{-1}f(x)b(x)$ is constant on $X$. 

If $g$ is undistorted, $\Delta_X(\bar{x})$ has strong $g$-specification and the cocycles $c_f^{(g),\pm}:\Delta_X\to H$ in Lemma \ref{limit exists in c_f^g, +} are equal. Then $b$ can be taken to be H\"{o}lder continuous.
\end{lem}
\begin{proof}
Observe that the proof of \cite[Lemma 6]{CJ} also works for $H\in\mathscr{G}_{inv}$, so we are left to check the continuous map $b$ defined there is H\"{o}lder continuous under the assumptions in the second part.

From the proof of \cite[Lemma 6]{CJ}, we know that for any $x, x'\in \Delta_X(\bar{x})\cap g^{-1}\Delta_X(\bar{x})$, 
\begin{eqnarray*}
d(b(x), b(x'))=d(b(x)b(x')^{-1}, e_H)=d(c_f^{(g), +}(x, x'), e_H). 
\end{eqnarray*}

Since $g$ is undistorted, from Proposition \ref{equivalence of undistorted g, linear (inverse) distortion function}, we know $\rho_g(n)\geq \lambda n$ for some $\lambda>0$ and all $n\geq 1$. Hence, $\rho_g^{-1}(n)\leq \frac{n}{\lambda}$.

Now, for any $R\in \mathbb{N}$, let $N=\lceil\frac{s'}{\lambda} \rceil \ell_S(g)4R+2R+\lceil t'\rceil+1\geq s'\ell_S(g)\frac{4R}{\lambda}+2R+t'+1 \geq\lceil s'\ell_S(g)\rho_g^{-1}(4R)+2R+t'\rceil$, from the estimate in the proof of Lemma \ref{compare lemma 5 in CJ}, 
we know that if $x_{B(N)}=x'_{B(N)}$,
then $d(c_f^{(g), +}(x, x'), e_H)\leq 2Cr^R\sum_{j\geq 0}r^{[r_j]}$. Here $r, C$ are constants appeared when defining H\"{o}lder continuity of $f$.

Hence, if $x_{B(N)}=x'_{B(N)}$, then 
$d(c_f^{(g), +}(x, x'), e_H)\leq  C'r'^N$, where $$C':=2C\sum_{j\geq 0}r^{[r_j]-\frac{\lceil t'\rceil+1}{2+4\ell_S(g)\lceil s'/\lambda\rceil}}>0, 0<r'=r^{\frac{1}{2+4\ell_S(g)\lceil s'/\lambda\rceil}}<1.$$ 

Therefore, $d(b(x), b(x'))\leq C'r'^N$ for $N=\lceil\frac{s'}{\lambda} \rceil \ell_S(g)4R+2R+\lceil t'\rceil+1$. Let $R$ change, then we deduce $d(b(x), b(x'))\leq C''r''^N$ for all $N\in \mathbb{N}$ and some constants $C''>0$, $0<r''<1$.

Then since $\Delta_X(\bar{x})\cap g^{-1}\Delta_X(\bar{x})$ is dense in $X$ and $b$ is continuous, a simple density argument implies that for all $x, x'\in X$ with $x_{B(N)}=x'_{B(N)}$, $d(b(x), b(x'))\leq C''r''^N$ for all $N\in \mathbb{N}$. Hence $b$ is H\"{o}lder continuous.
\end{proof}

\begin{lem}
\label{compare lemma 7 in CJ}
Assume $G$ has one end, $g_i\in G$ is an element with (SDT) property, and there exists a point $\bar{x}\in X$ such that $\Delta_X(\bar{x})\cap {g_i}^{-1}\Delta_X(\bar{x})$ is dense in $X$ and $$\lim_{\substack{(x,x')\to\Delta \\ x,x'\in\Delta_X(\bar{x})}}d(c_{f_{g_i}}^{(g_i),+}(x,x'),e_H)=0,$$ for all $i=1,\cdots, k$, where $k$ is any positive integer (maybe infinity), then there is a  continuous map $b: X\to H$ such that the map $x\mapsto b(gx)^{-1}c(g, x)b(x)$ is constant on $X$ (depending only on $g$) for all $g\in \langle g_1, \cdots, g_k \rangle$, where $c: G\times X\to H$ is a H\"{o}lder continuous cocycle and $H\in\mathscr{G}_{inv}$. If we further assume $g_i$ is undistorted and $\Delta_X(\bar{x})$ has strong $g_i$-specification for all $i$, then such a $b$ can be taken to be H\"{o}lder continuous.
\end{lem}

\begin{proof}
The proof is the same as the proof of \cite[Lemma 7]{CJ} after replacing \cite[Lemma 6, Lemma 2]{CJ} with Lemma \ref{compare lemma 6 in CJ} and Lemma \ref{compare lemma 2 in CJ} respectively.
\end{proof}

\begin{lem}
\label{compare lemma 1 in CJ}
The statement is the same as \cite[Lemma 1]{CJ} except we replace ``continuous" by ``H\"{o}lder continuous" everywhere.
\end{lem}
\begin{proof}
The proof still works here since we do not change the transfer map.
\end{proof}
Using these Lemmas, we can finish the proof of Theorem \ref{main thm in 2nd approach} and Corollary \ref{corollary on the membership of G_h}.
\begin{proof}[Proof of Theorem \ref{main thm in 2nd approach}]
Same proof as the proof of \cite[Theorem 1]{CJ} still works here. Note that we need to use Lemma \ref{compare lemma 8 in CJ} now instead of \cite[Lemma 8]{CJ}.
\end{proof}
\begin{proof}[Proof of Corollary \ref{corollary on the membership of G_h}]
By Lemma \ref{full shifts have strong specification}, we can apply Theorem \ref{main thm in 2nd approach} to full shifts. 
\end{proof}

\section{Examples, groups in $\cG_H$ and remarks}\label{last section}
Besides full shifts, we give another class of subshifts satisfying assumptions of Theorem \ref{main thm in 2nd approach}. 
\begin{example}(\textbf{The generalized golden mean subshifts}) Let $A=\{0,1,\cdots, k\}$ and $F_1,\cdots, F_m$ be non-empty finite subsets of $G$. Let $X(F_1,\cdots, F_m)$ be a subset of $A^G$ consisting of all $x\in A^G$ such that for every $1\leq j\leq m$, $g\in G$, there exists $g_j\in gF_j$ such that $x_{g_j}=0$. Then it is a subshift of $A^G$. Using the same argument as in the proof of \cite[Lemma 10]{CJ}, we also obtain that $\Delta_X$ has strong $a$-specification in the sense of definition \ref{D-specification} for every element $a$ in $G$ with infinite order. And hence similar to the proof of \cite[Corollary 4]{CJ}, we know that the generalized golden mean subshifts $X(F_1,\cdots, F_m)$ are examples of our theorem \ref{main thm in 2nd approach} 
\end{example}
\begin{remark}
In our main theorem, if we drop the condition of specification, the result does not hold any more even for the special case $G=\Z^d$, $d>1$, see for example \cite[Example 4.3]{Schmidt95}. 
\end{remark}
Now we will investigate the class of one-ended groups having undistorted elements and sub-exponential growing divergence functions. Before giving examples of groups with sub-exponential growing divergence functions, let us review definitions of wide groups, unconstricted groups, and thick groups \cite{bd,BDM,ds}.

Let $X$ be a connected, locally connected topological space. A point $x\in X $ is a (global) cut-point if $X\setminus\{x\}$ has at least two connected components.
\begin{definition}
We call a finitely generated group $G$ unconstricted if one of its asymptotic cones has no cut-points.
We say a finitely generated group $G$ wide if none of its asymptotic cones has a cut-point. 
\end{definition}
Clearly, wide groups are unconstricted. And by Stallings' Ends theorem, we know that unconstricted groups are one-ended groups. Recall that a finitely generated subgroup $H$ of a finitely generated group $G$ is undistorted if any word metric of $H$ is bi-Lipschitz equivalent to a word metric of $G$ restricted to $H$.
\begin{definition}
(Algebraic network of subgroups) Let $G$ be a finitely generated group,
$\cH$ be a finite collection of subgroups of $G$ and let $M > 0$. We say $G$ is an
$M$-algebraic network with respect to $\cH$ if
\begin{itemize}
\item All subgroups in $\cH$ are finitely generated and undistorted in $G$;
\item there is a finite index subgroup $G_1$ of $G$ such that $G \subset N_M(G_1)$ and such
that a finite generating set of $G_1$ is contained in $\bigcup_{H\in \cH}H$;
\item any two subgroups $H, H'$ in $\cH$ can be thickly connected in H in the sense that there exists
a finite sequence $H = H_1,\dots,H_n = H'$ of subgroups in $\cH$ such that for all $1 \leq i < n$, $H_i \cap H_{i+1}$ is infinite.
\end{itemize}
\end{definition}
\begin{definition}
(Algebraic thickness) Let $G$ be a finitely generated group. $G$ is called algebraically thick of order zero if it is unconstricted. We say $G$ is $M$-algebraically thick of order at most $n + 1$ with respect to $\cH$, where $\cH$ is a finite collection of subgroups of $G$ and $M > 0$, if
\begin{itemize}
\item $G$ is an $M$-algebraic network with respect to $H$;
\item all subgroups in $\cH$ are algebraically thick of order at most $n$.
\end{itemize}
$G$ is said to be algebraically thick of order $n + 1$ with respect to $\cH$, when $n$ is the smallest value for which this statement holds. And $G$ is algebraically thick of order $n + 1$ if $G$ is algebraically thick of order $n + 1$ with respect to $\cH$ for some $\cH$.
\end{definition}
The algebraic thickness property of $G$ does not depend on the word metric on $G$.

Let $X$ be a metric space. For $L,C>0$, an $(L,C)$-quasi-geodesic is an $(L,C)$-quasi-isometric embedding $f: I \to X,$ where $I$ is a connected subset in $\R$. For every subset $A$ of $X$ and $r>0$, we denote by $N_r(A):=\{x\in X:dist(x,A)<r\}$. For $C>0$, a subset $A$ in $X$ is called $C$-path connected if any two points in $A$ can be connected by a path in $N_C(A)$. We say that $A$ is $C$-quasi-convex
if any two points in $A$ can be connected in $N_C(A)$ by an $(C,C)$-quasi-geodesic. 
\begin{definition} (Tight algebraic network of subgroups). We say a finitely
generated group $G$ is an $M$-tight algebraic network with respect to $\cH$ if 
\begin{itemize}
\item $\cH$ is a collection of $M$-quasiconvex
subgroups
\item there is a finite index subgroup $G_1$ of $G$ such that $G \subset N_M(G_1)$ and such
that a finite generating set of $G_1$ is contained in $\bigcup_{H\in \cH}H$;
\item and for any two subgroups $H,H' \in \cH$ there exists a finite sequence $H = H_1, \dots , H_m=H'$ of subgroups in $\cH$ such that for all $1\leq  i < m$, the intersection $H_i \cap H_{i+1}$ is infinite and $M$-path connected. 
\end{itemize}
\end{definition}
\begin{definition} (Strong algebraic thickness). Let $G$ be a finitely generated group. $G$ is strongly algebraically thick of order zero if it is wide. Given $M >0$, we say G is $M$-strongly algebraically thick of order
at most $n+1$ with respect to a finite collection of subgroups $\cH$ of $G$, if
\begin{itemize}
\item $G$ is an $M$-tight algebraic network with respect to $\cH$;
\item all subgroups in $\cH$ are strongly algebraically thick of order at most $n$.
\end{itemize}
We say $G$ is strongly algebraically thick of order $n+1$ with respect to $\cH$, when $n$ is the smallest value for which this statement holds. And $G$ is strongly algebraically thick of order $n + 1$ if $G$ is strongly algebraically thick of order $n + 1$ with respect to $\cH$ for some $\cH$.
\end{definition}
The strongly algebraic thickness property of $G$ does not depend on the word metric on $G$.

We list some examples of $G$ with sub-exponential growing divergence functions. 

\begin{itemize}
\item linear divergence: From \cite[Proposition 1.1]{dms}, every group is wide if and only if it has linear divergence. And here are some examples of wide groups.
\begin{enumerate}
\item Non-virtually cyclic groups satisfying a law are wide \cite[Corollary 6.13]{ds}. A law is a word $w$ in $n$ letters $x_1, \dots , x_n$ and a group $G$ satisfies the law $w$ if $w = 1$ in $G$ whenever
$x_1,\dots, x_n$ are replaced by every set of $n$ elements in $G$. Solvable groups, uniformly amenable finitely generated groups are groups satisfying a law \cite[Corollary 6.16]{ds}.
\item Non-virtually cyclic groups whose center contains $\Z$ are wide \cite[Theorem 6.5]{ds}. 
\item Products of arbitrary infinite groups are wide \cite[Example 1, page 555]{BDM}. Note that although \cite[Example 1, page 555]{BDM} only mentioned that these groups are unconstricted, indeed from there we can get they are actually wide.
\item Let $G$ be an irreducible lattice in a semi-simple Lie group of higher rank. Assume that $G$ is either of $\Q$-rank 1 or is of the form $SL_n(\cO_S)$, where $n\geq 3$, $S$ is a finite set of valuations of a number field $K$ including all infinite valuations, and $\cO_S$ is the associated ring of $S$-integers. Then $G$ is wide \cite[Proposition 1.1 and Theorem 1.4]{dms}.
\item \cite[page 2455, last paragraph]{dms} Let $G$ be a connected semisimple Lie group with finite center, no nontrivial compact factors and $\R$-rank$>1$. Let $\Gamma$ be a uniform lattice of $G$. Then $\Gamma$ is quasi-isometric to $G$. And hence each asymptotic cone of $\Gamma$ is bi-Lipschitz equivalent to an asymptotic cone of $G$. With conditions of $G$, the quotient space $G/K$ where $K$ is a maximal compact subgroup of $G$ is a symmetric space of noncompact type and has no Euclidean factor. Thus, applying \cite[Theorem 5.2.1]{KL}, we get that every asymptotic cone of $G/K$ is an Euclidean building of rank$>1$ and every two points in the asymptotic cone belong to a 2-dimensional flat. Therefore, every asymptotic cone of $G/K$ has no cut points. Hence so are $G$ and $\Gamma$. Thus, $\Gamma$ is wide.   
\item R. Thompson groups $F:=\langle s,t|[st^{-1},s^{-1}ts]=[st^{-1},s^{-2}ts^2]=1\rangle$, $T$, $V$ \cite[Corollary 2.14]{GolanSapir}.
\end{enumerate}
\item polynomial growth divergence: if a finitely generated group $G$ is strongly algebraic thick
of order $n$, then its divergence is at most polynomial of degree $n + 1$ \cite[Corollary 4.17]{bd}. And here are some examples of strongly algebraic thick groups. Note that although in \cite{BDM} the authors only proved that these groups are algebraic thick , the proofs there also imply that they are actually strongly algebraic thick groups.
\begin{enumerate}
\item Mapping class groups $MCG(\Sigma_{g,p})$, where $\Sigma_{g,p}$ is an orientable surface of genus $g$ with $p$ punches satisfying $3g+p>4$ \cite[Theorem 8.1]{BDM};
\item $Aut(F_n), Out(F_n)$ for $n\geq 3$ \cite[Theorem 9.2]{BDM};
\end{enumerate}
\item sub-exponential growth: in \cite[Corollary 6.4]{OOS}, the authors constructed groups whose divergences are sub-exponential but non-polynomials. However, these groups are torsion and hence do not have undistorted elements.
\end{itemize}
Non-elementary word hyperbolic groups are examples of groups with at least exponential divergence, see \cite[Theorem 2.19]{al_8} or \cite{gromov_book on hyperbolic groups}.

Many groups satisfy property (UD). Here we list some of them: 
\begin{itemize}
\item finitely generated abelian groups \cite[Lemma 6.4]{GS}.
\item Heisenberg group $H_3(\mathbb{Z})$.
\item biautomatic groups \cite[Proposition 6.6]{GS}. Geometrically finite hyperbolic groups are biautomatic \cite[Theorem 11.4.1]{ECHLPT}
\item mapping class group $MCG(\Sigma_{g,m})$, where $\Sigma_{g,m}$ is an oriented surface of genus $g$ and $m$ punches \cite[Theorem 1.1]{FLM}.
\item (outer) automorphism group of a finitely generated free group \cite[Theorem 1.1]{Ali}.
\item semihyperbolic groups \cite[III. Lemma 4.18, p. 479]{BH}. Note that biautomatic groups are semihyperbolic \cite[page 59]{Farb}. 
\item Let $\Gamma$ be an irreducible uniform lattice in a semisimple linear Lie group $G\neq SO(n,1)$. Then $\Gamma$ is bicombable, i.e. $\Gamma$ is semihyperbolic \cite[Theorem 1.1 and page 59]{Farb}. Indeed, following the remark after \cite[Theorem 1.1]{Farb}, we can prove that every irreducible uniform lattice in a connected semisimple Lie group with finite center is also semihyperbolic.
\item $SL_n(\mathbb{Z})$ \cite[Examples 2.13--2.15]{CF}.
\item Let $(M;\omega)$ be a closed symplectically hyperbolic manifold and denote by $Ham(M;\omega)$ the group
of Hamiltonian diffeomorphisms. Then every finitely generated subgroup $G$ of $Ham(M;\omega)$ has undistorted elements \cite[Theorem 1.6.A and remark 1.6.C]{Polterovich}
\item R. Thompson groups $F$ \cite[Proposition 4]{Burillo}, $T$ \cite[Corollary 5.4]{BCST}, $V$ \cite[Theorem 1.3]{MR3134027}  (note that the R. Thompson group $V$ is $V_2$ in \cite{MR3134027}).
\end{itemize}
From the above arguments, we know that the class of one-ended groups having undistorted elements and sub-exponential growing divergence functions will contain the class (UW) and the class (ET), where (UW) is the class of wide groups having undistorted elements and (ET) is the class of one-ended groups being strongly algebraic thick of order $n$ for some $n\geq 1$.

The following groups are in the class (UW):
\begin{itemize}
\item Let $G$ be a connected semisimple higher rank Lie groups with finite center and no nontrivial compact factors. Let $\Gamma$ be an irreducible uniform lattice in $G$. Then $\Gamma$ is wide and has undistorted elements.
\item Baumslag-Solitar groups $BS(m, n)$ with $1<|m|=|n|$.
\item Heisenberg group $H_3(\mathbb{Z})$.
\item $\Z\times G$, where $G$ is infinite.
\item $SL_n(\Z)$, $n\geq 3$.
\item R. Thompson groups $F,T,V$.
\end{itemize}

To see $BS(m, n)\in (UW)$, observe that it has a finite index subgroup $\Z\times F_n$ which has one end, linear divergence and (UD) when $1<|m|=|n|$ and the following holds.

\begin{lem}
Let $G'<G$ be a subgroup inclusion of finitely generated groups. If the inclusion map is a quasi-isometry (e.g. $[G: G']<\infty$), and $G'$ has (UD) property, then $G$ also has (UD) property.
\end{lem}
The class of (ET) includes the following groups
\begin{enumerate}
\item $Aut(F_n), Out(F_n)$ for $n\geq 3$. 
\item Mapping class groups $MCG(\Sigma_{g,p})$, where $\Sigma_{g, p}$ is a closed, orientable, connected surface of genus $g>1$ with $p$-punches.
\end{enumerate}

Applying \cite[Section 3]{CV}, we know that $Aut(F_n), Out(F_n)$ and $MCG(\Sigma_{g,p})$ for $n\geq 3$, $g>1$, $p\geq 0$ have property $F\R$, i.e every its action on an $\R$-tree have a fixed point. As a consequence, they have property FA and hence applying \cite[Theorem 15, page 58]{Serre} and Stallings' End theorem we know that they are one-ended groups.

From \cite[Corollary 6.4]{OOS}, we ask the following question
\begin{question}
Are there finitely generated groups whose divergences are sub-exponential but non-polynomials have undistorted elements?
\end{question}

Our class of groups $\cG_H$ also contains all one-ended right-angled Artin groups. Let  $\Gamma$ be a finite, simplicial graph with vertex set $V$. The right-angled Artin group (RAAG) associated to $\Gamma$ is the group $A_\Gamma$ with presentation
$$A_\Gamma: =\{V ~|~ vw = wv \mbox { if $v$ and $w$ are connected by an edge in }\Gamma\}.$$
If $\Gamma$ is connected then $A_\Gamma$ is always has polynomial divergence \cite[Corollary 4.8]{BC}. Furthermore if $\Gamma$ has at least 3 vertices then $A_\Gamma$ has one end when $\Gamma$ is connected \cite[Theorem B]{BM} or \cite[Corollary 5.2]{Meier}. Because $A_\Gamma$ projects onto $\Z^n$, where $n$ is the number of vertices in $\Gamma$, by adding all commuting relations between the generators, for any generator $s$ in $A_\Gamma$, the length of $s^k$ in $A_\Gamma$ is greater than or equal to the length of $\bar{s}^k$ in $\Z^n$. Thus, $A_\Gamma$ is undistorted. Therefore, if $V$ has at least 3 vertices and is connected then the RAAG group $A_\Gamma$ belongs to $\cG_H$.

Similarly, certain right-angled Coxeter groups (RACGs) also belong to our class. Given a finite simplicial graph $\Gamma$ with vertices $V$, the associated RACG $W_\Gamma$ is the group with presentation 
$$W_\Gamma:=\{V~|~ v^2=1, vw = wv \mbox { if $v$ and $w$ are connected by an edge in V}\}.$$ 
For a graph $\Gamma$, we define the associated four-cycle graph $\Gamma^4$ as follows. The embedded loops of length four (i.e. four-cycles) in $\Gamma$ are the vertices of $\Gamma^4$. Two
vertices of $\Gamma^4$ are connected by an edge if the corresponding four-cycles in $\Gamma$ share a pair of adjacent edges. Given a subgraph $\Gamma_1$ of $\Gamma^4$, we define the support of $\Gamma_1$ to be the collection of vertices
of $\Gamma$ appearing in the four-cycles in $\Gamma$ corresponding to the vertices of $\Gamma_1$. A graph $\Gamma$ is said to be CFS (Component with
Full Support) if there exists a component of $\Gamma^4$ whose support is the entire vertex set of $\Gamma$. A point or vertex in $\Gamma$ is separating if its complement in $\Gamma$ is not connected. If $\Gamma$ is connected, triangle-free; has no separating vertices or edges then $W_\Gamma$ is one-ended \cite[Theorem 8.7.2]{Davis}. Furthermore, if it is join, i.e a complete bipartite graph, or CFS then $W_\Gamma$ has a polynomial divergence \cite[Theorem 1.1]{DaniThomas}. 
Besides, if $W_\Gamma$ is not finite, i.e. there exist two vertices, say $v$ and $w$, of $\Gamma$ that are not connected by an edge, then the element $g:=vw$ is undistorted in $W_\Gamma$ by applying Theorems 3.2.16, 3.3.4, 3.4.2 and Corollary 6.12.6 in \cite{Davis}.


Since we are unaware of any criterion to classify one-ended groups with both sub-exponential divergence function and property (UD), we want to end the paper with the following question:
\begin{question}
Let $G$ be a finitely generated non-torsion amenable group that is not virtually cyclic, does $G$ have (UD)? Does $G$ have sub-exponential divergence function? What about $G$ is any finitely generated infinite group with property (T) but not hyperbolic?
\end{question}

\begin{bibdiv}
\begin{biblist}
\bib{Ali}{article}{
   author={Alibegovi\'c, E.},
   title={Translation lengths in ${\rm Out}(F_n)$},
   journal={Geom. Dedicata},
   volume={92},
   date={2002},
   pages={87--93},
}  
\bib{al_8}{article}{
   author={Alonso, J. M.},
   author={Brady, T.},
   author={Cooper, D.},
   author={Ferlini, V.},
   author={Lustig, M.},
   author={Mihalik, M.},
   author={Shapiro, M.},
   author={Short, H.},
   title={Notes on word hyperbolic groups},
   note={Edited by Short},
   conference={
      title={Group theory from a geometrical viewpoint},
      address={Trieste},
      date={1990},
   },
   book={
      publisher={World Sci. Publ., River Edge, NJ},
   },
   date={1991},
   pages={3--63},}

\bib{BC}{article}{
   author={Behrstock, J.},
   author={Charney, R.},
   title={Divergence and quasimorphisms of right-angled Artin groups},
   journal={Math. Ann.},
   volume={352},
   date={2012},
   number={2},
   pages={339--356},
  
}
		



\bib{bd}{article}{
   author={Behrstock, J.},
   author={Dru\c tu, C.},
   title={Divergence, thick groups, and short conjugators},
   journal={Illinois J. Math.},
   volume={58},
   date={2014},
   number={4},
   pages={939--980},
   }
\bib{BDM}{article}{
   author={Behrstock, J.},
   author={Dru\c tu, C.},
   author={Mosher, L.},
   title={Thick metric spaces, relative hyperbolicity, and quasi-isometric
   rigidity},
   journal={Math. Ann.},
   volume={344},
   date={2009},
   number={3},
   pages={543--595},
  
}
\bib{MR3134027}{article}{
   author={Bleak, C.},
   author={Bowman, H.},
   author={Gordon Lynch, A.},
   author={Graham, G.},
   author={Hughes, J.},
   author={Matucci, F.},
   author={Sapir, E.},
   title={Centralizers in the R. Thompson group $V_n$},
   journal={Groups Geom. Dyn.},
   volume={7},
   date={2013},
   number={4},
   pages={821--865},

   }
\bib{BM}{article}{
   author={Brady, N.},
   author={Meier, J.},
   title={Connectivity at infinity for right angled Artin groups},
   journal={Trans. Amer. Math. Soc.},
   volume={353},
   date={2001},
   number={1},
   pages={117--132},
   
}
		
	

\bib{BH}{book}{
   author={Bridson, M. R.},
   author={Haefliger, A.},
   title={Metric spaces of non-positive curvature},
   series={Grundlehren der Mathematischen Wissenschaften [Fundamental
   Principles of Mathematical Sciences]},
   volume={319},
   publisher={Springer-Verlag, Berlin},
   date={1999},
  }
   
\bib{Burillo}{article}{
   author={Burillo, J.},
   title={Quasi-isometrically embedded subgroups of Thompson's group $F$},
   journal={J. Algebra},
   volume={212},
   date={1999},
   number={1},
   pages={65--78},
  
}
\bib{BCST}{article}{
   author={Burillo, J.},
   author={Cleary, S.},
   author={Stein, M.},
   author={Taback, J.},
   title={Combinatorial and metric properties of Thompson's group $T$},
   journal={Trans. Amer. Math. Soc.},
   volume={361},
   date={2009},
   number={2},
   pages={631--652},
 }
	   

\bib{CF}{article}{
   author={Calegari, D.},
   author={Freedman, M. H.},
   title={Distortion in transformation groups},
   note={With an appendix by Yves de Cornulier},
   journal={Geom. Topol.},
   volume={10},
   date={2006},
   pages={267--293},
   }



\bib{CJ}{article}{
   author={Chung, N.},
   author={Jiang, Y.},
   title={Continuous Cocycle Superrigidity for Shifts and Groups with One End},
  journal={Math. Ann.},
   volume={368},
   date={2017},
   number={3-4},
   pages={1109--1132},
 }

\bib{cohen}{article}{
author={Cohen, D. B.},
title={Continuous cocycle superrigidity for the full shift over a finitely generated torsion group},
status={arXiv: 1706.03743},
}

\bib{CV}{article}{
   author={Culler, M.},
   author={Vogtmann, K.},
   title={A group-theoretic criterion for property ${\rm FA}$},
   journal={Proc. Amer. Math. Soc.},
   volume={124},
   date={1996},
   number={3},
   pages={677--683},
   }
   
\bib{DaniThomas}{article}{
   author={Dani, P.},
   author={Thomas, A.},
   title={Divergence in right-angled Coxeter groups},
   journal={Trans. Amer. Math. Soc.},
   volume={367},
   date={2015},
   number={5},
   pages={3549--3577},
  
}   
\bib{Davis}{book}{
   author={Davis, M. W.},
   title={The geometry and topology of Coxeter groups},
   series={London Mathematical Society Monographs Series},
   volume={32},
   publisher={Princeton University Press, Princeton, NJ},
   date={2008},
   
}

  \bib{dms}{article}{
   author={Dru\c tu, C.},
   author={Mozes, S.},
   author={Sapir, M.},
   title={Divergence in lattices in semisimple Lie groups and graphs of
   groups},
   journal={Trans. Amer. Math. Soc.},
   volume={362},
   date={2010},
   number={5},
   pages={2451--2505},}

\bib{ds}{article}{
   author={Dru\c tu, C.},
   author={Sapir, M.},
   title={Tree-graded spaces and asymptotic cones of groups},
   note={With an appendix by Denis Osin and Mark Sapir},
   journal={Topology},
   volume={44},
   date={2005},
   number={5},
   pages={959--1058},}      

\bib{ECHLPT}{book}{
   author={Epstein, D. B. A.},
   author={Cannon, J. W.},
   author={Holt, D. F.},
   author={Levy, S. V. F.},
   author={Paterson, M. S.},
   author={Thurston, W. P.},
   title={Word processing in groups},
   publisher={Jones and Bartlett Publishers, Boston, MA},
   date={1992},
   } 
\bib{Farb}{article}{
   author={Farb, B.},
   title={Combing lattices in semisimple Lie groups},
   conference={
      title={Groups---Korea '94 (Pusan)},
   },
   book={
      publisher={de Gruyter, Berlin},
   },
   date={1995},
   pages={57--67},
 
} 

\bib{FLM}{article}{
   author={Farb, B.},
   author={Lubotzky, A.},
   author={Minsky, Y.},
   title={Rank-1 phenomena for mapping class groups},
   journal={Duke Math. J.},
   volume={106},
   date={2001},
   number={3},
   pages={581--597}, 
}
 \bib{FH}{article}{
   author={Franks, J.},
   author={Handel, M.},
   title={Distortion elements in group actions on surfaces},
   journal={Duke Math. J.},
   volume={131},
   date={2006},
   number={3},
   pages={441--468},
  }
   
\bib{Furman}{article}{
   author={Furman, A.},
   title={On Popa's cocycle superrigidity theorem},
   journal={Int. Math. Res. Not. IMRN},
   date={2007},
   number={19},
   pages={Art. ID rnm073, 46},
  
}
\bib{gersten_div}{article}{
   author={Gersten, S. M.},
   title={Quadratic divergence of geodesics in ${\rm CAT}(0)$ spaces},
   journal={Geom. Funct. Anal.},
   volume={4},
   date={1994},
   number={1},
   pages={37--51},} 

 \bib{GS}{article}{
   author={Gersten, S. M.},
   author={Short, H. B.},
   title={Rational subgroups of biautomatic groups},
   journal={Ann. of Math. (2)},
   volume={134},
   date={1991},
   number={1},
   pages={125--158},}
   
\bib{GolanSapir}{article}{
author={Golan, G.},
author={Sapir, M.},
title={Divergence of Thompson groups},
status={preprint,  arXiv:1709.08144}
}	   

   
\bib{gromov_book on hyperbolic groups}{article}{
   author={Gromov, M.},
   title={Hyperbolic groups},
   conference={
      title={Essays in group theory},
   },
   book={
      series={Math. Sci. Res. Inst. Publ.},
      volume={8},
      publisher={Springer, New York},
   },
   date={1987},
   pages={75--263},}  
    
\bib{gromov_book on asymptotic invariants}{article}{
   author={Gromov, M.},
   title={Asymptotic invariants of infinite groups},
   conference={
      title={Geometric group theory, Vol.\ 2},
      address={Sussex},
      date={1991},
   },
   book={
      series={London Math. Soc. Lecture Note Ser.},
      volume={182},
      publisher={Cambridge Univ. Press, Cambridge},
   },
   date={1993},
   pages={1--295},}
\bib{KatokNitica}{book}{
   author={Katok, A.},
   author={Ni\c tic\u a, V.},
   title={Rigidity in higher rank abelian group actions. Volume I},
   series={Cambridge Tracts in Mathematics},
   volume={185},
   publisher={Cambridge University Press, Cambridge},
   date={2011},
  
}

\bib{katok_schmidt}{article}{
   author={Katok, A. B.},
   author={Schmidt, K.},
   title={The cohomology of expansive ${\bf Z}^d$-actions by automorphisms
   of compact, abelian groups},
   journal={Pacific J. Math.},
   volume={170},
   date={1995},
   number={1},
   pages={105--142},}
   
\bib{KatokSpatzier}{article}{
   author={Katok, A.},
   author={Spatzier, R. J.},
   title={First cohomology of Anosov actions of higher rank abelian groups
   and applications to rigidity},
   journal={Inst. Hautes \'Etudes Sci. Publ. Math.},
   number={79},
   date={1994},
   pages={131--156},
  
}
\bib{KL}{article}{
   author={Kleiner, B.},
   author={Leeb, B.},
   title={Rigidity of quasi-isometries for symmetric spaces and Euclidean
   buildings},
   journal={Inst. Hautes \'Etudes Sci. Publ. Math.},
   number={86},
   date={1997},
   pages={115--197 (1998)},
  
}

\bib{xinli}{article}{
author={Li, X.},
title={Continuous orbit equivalence rigidity},
status={to appear in Ergod. Th. Dyn. Sys.}
}	
\bib{Meier}{article}{
   author={Meier, J.},
   title={Geometric invariants for Artin groups},
   journal={Proc. London Math. Soc. (3)},
   volume={74},
   date={1997},
   number={1},
   pages={151--173},
   
}
\bib{OOS}{article}{
   author={Ol\cprime shanskii, A. Y.},
   author={Osin, D. V.},
   author={Sapir, M. V.},
   title={Lacunary hyperbolic groups},
   note={With an appendix by Michael Kapovich and Bruce Kleiner},
   journal={Geom. Topol.},
   volume={13},
   date={2009},
   number={4},
   pages={2051--2140},
  
}
\bib{Popa2}{article}{
   author={Popa, S.},
   title={Cocycle and orbit equivalence superrigidity for malleable actions
   of $w$-rigid groups},
   journal={Invent. Math.},
   volume={170},
   date={2007},
   number={2},
   pages={243--295},
  }
\bib{Popa3}{article}{
   author={Popa, S.},
   title={On the superrigidity of malleable actions with spectral gap},
   journal={J. Amer. Math. Soc.},
   volume={21},
   date={2008},
   number={4},
   pages={981--1000},
   }  
   
\bib{Polterovich}{article}{
   author={Polterovich, L.},
   title={Growth of maps, distortion in groups and symplectic geometry},
   journal={Invent. Math.},
   volume={150},
   date={2002},
   number={3},
   pages={655--686},
  
}   



\bib{Schmidt95}{article}{
   author={Schmidt, K.},
   title={The cohomology of higher-dimensional shifts of finite type},
   journal={Pacific J. Math.},
   volume={170},
   date={1995},
   number={1},
   pages={237--269},
}

\bib{Serre}{book}{
   author={Serre, J.-P.},
   title={Trees},
   note={Translated from the French by John Stillwell},
   publisher={Springer-Verlag, Berlin-New York},
   date={1980},
 }

\bib{walters}{book}{
   author={Walters, P.},
   title={An introduction to ergodic theory},
   series={Graduate Texts in Mathematics},
   volume={79},
   publisher={Springer-Verlag, New York-Berlin},
   date={1982},
   
}   

\end{biblist}
\end{bibdiv}
\end{document}